\documentclass[10pt]{amsart}
\usepackage{amsmath}
\usepackage[usenames,dvipsnames]{color}
\usepackage{parskip}
\usepackage{amsfonts}
\usepackage{amscd}
\usepackage[centertags]{amsmath}
\usepackage{amssymb}
\usepackage{stmaryrd}
\usepackage[all,cmtip]{xy}
\usepackage[english]{babel}
\usepackage{tabularx}
\usepackage{amsxtra}
\usepackage{euscript}
\usepackage[T1]{fontenc}
\usepackage{doc, exscale, fontenc, latexsym, syntonly}
\usepackage{amsfonts}
\usepackage{amsthm}
\usepackage{graphicx}
\usepackage{mciteplus}
\usepackage{cite}

\newtheorem{theorem}{Theorem}[section]

\newtheorem{proposition}[theorem]{Proposition}

\theoremstyle{definition}
\newtheorem{definition}[theorem]{Definition}
\newtheorem{example}[theorem]{Example}
\theoremstyle{remark}

\numberwithin{figure}{section}
\numberwithin{table}{section}
\begin{document}

\title[Digital $n-$Manifolds With Or Without Boundaries]{Digital $n-$Manifolds With Or Without Boundaries}

\author{MEL\.{I}H \.{I}S}
\date{\today}

\address{\textsc{Melih İs,}
Ege University\\
Faculty of Sciences\\
Department of Mathematics\\
İzmir, Türkiye}
\email{melih.is@ege.edu.tr}

\author{\.{I}SMET KARACA}
\date{\today}

\address{\textsc{İsmet Karaca,}
	Ege University\\
	Faculty of Sciences\\
	Department of Mathematics\\
	İzmir, Türkiye}
\email{ismet.karaca@ege.edu.tr}



\begin{abstract}
       This work aims to define the concept of manifold, which has a very important place in the topology, on digital images. So, a general perspective is provided for two and three-dimensional imaging studies on digital curves and digital surfaces. Throughout the study, the features present in topological manifolds but that are not satisfied in the discrete version are specifically underlined. In addition, other concepts closely related to manifolds such as submanifold, orientation, and partition of unity are also discussed in digital images.  
\end{abstract}

\maketitle

%%%%%%%%%%%%%%%%%%%%%%%%%%%%%%%%%%%%%%%%%%%%%%%%%%%%%
\section{Introduction}
\label{intro}
\quad A topological space locally homeomorphic to $n-$dimensional Euclidean space is called a topological $n-$manifold. It provides for a large variety of global structures while having a local resemblance to the Euclidean space. As a fundamental idea in modern mathematics and physics, a manifold takes the intuitive idea of curves and surfaces to higher dimensions \cite{Lee:2010}. The concept of a manifold is first proposed by Bernhard Riemann in the 19th century and differentiable manifolds are his main area of interest, especially in the setting of geometry and analysis \cite{Riemann:1854}. As topology advanced in the early 20th century, the definition of topological manifolds is presented by the work of mathematicians like Henri Poincaré and Felix Hausdorff \cite{Hausdorff:1914,Poincare:1895}.

\quad Topological manifolds have some essential characteristics that make them essential to science and mathematics. Depending on the situation, they can impose differentiable, Riemannian, or symplectic structures, which is one of their most remarkable characteristics. As an example in physics, especially in general relativity, spacetime is represented as a four-dimensional differentiable manifold with a Lorentzian metric. For another, data science and manifold studies are related, as manifold learning techniques seek to reveal low-dimensional manifold structures hidden behind high-dimensional data \cite{TenenSilLang:2000}. 

\quad Manifolds facilitate the simultaneous examination of both local and global properties. For example, the fact that a manifold is locally flat (differentiable) allows some analysis operations such as differentiation and integration to be performed on it. However, the global topological structure of the manifold dictates the fundamental behavior of these operations. For instance, a $2-$dimensional sphere ($S^{2}$) shares the local properties of the plane ($\mathbb{R}^{2}$) but has a different topological structure globally. In algebraic topology and homotopy theory, invariants such as homology and cohomology groups aid in the classification and structural analysis of manifolds \cite{Spanier:1966}. A manifold's holes, gaps, and other topological characteristics can be expressed numerically using these tools. As an example, the fundamental group of a manifold analyzes its cyclic structures, whereas homology groups offer a broader viewpoint. Such topological and algebraic topological information plays a critical role in the classification of manifolds.

\quad Digital topology investigates the application of traditional topological ideas to digital images and offers numerous significant results on digital curves and surfaces \cite{KongRos:1989}. These structures are so crucial, especially in the fields of digital image processing and computer graphics. Digital surfaces are employed in the structural study of three-dimensional models, and digital curves are frequently utilized to depict boundaries in digital images \cite{Ros:1970,Ros:1979,Herman:1998}. A digital curve is commonly described as a simple closed path \cite{Ros:1970}. According to the Jordan curve theorem in classical topology, a plane is divided into inner and exterior regions by a simple closed curve. This theorem, when applied to the digital plane in digital topology, demonstrates how a digital curve separates a picture into an interior and an exterior region \cite{KongRos:1989}. A digital curve is said to be simple if it simply has the same beginning and ending points and is not closed on itself. The accuracy of algorithms used in digital image processing depends on this feature \cite{Ros:1970}. On the other hand, in three-dimensional digital images, digital surfaces are used to simulate an object's interior structure or boundary \cite{Herman:1998}. The homotopic and homological properties of digital surfaces are similar to those of classical surfaces. Specifically, digital surfaces are analyzed using topological invariants such as the Euler characteristic \cite{Chen:2004}. Applications like volume calculation, boundary detection, and three-dimensional object reconstruction commonly use digital surfaces. Digital surface analysis is essential, particularly in the fields of industrial modeling and medical imaging \cite{Herman:1998}. These examinations on digital curves and surfaces demonstrate both the theoretical and practical value of digital topology.

\quad Historically, the first ideas about the concept of digital manifold, and the works of formalizing these ideas in a systematic order are carried out by Rosenfeld \cite{Ros:1970}, Herman \cite{Herman:1998}, \cite{Chen:1991}, and \cite{ChenZhang:1993}. However, these studies are generally based on parallel moves or cell structures and do not provide an explicit definition of digital manifold via adjacency relations. Mathematically, in this manuscript, we examine the elements that constitute the topological manifold, namely the properties of being Hausdorff, being second-countable, and being local homeomorphism, on digital images and how the concept of digital manifold should be defined using adjacency relations. In Section \ref{sec:1}, we examine how some important definitions and results are used in digital images. After that, in Section \ref{sec:2}, we present the notion of 'digital $n-$manifold' (and also 'digital $n-$manifold with boundary') by interpreting the discrete version of the local homeomorphism property (note that we shall explain the fact that this is enough to define a topological $n-$manifold in digital topology in the section). Furthermore, we immediately exemplify this new concept by considering the digital intervals, digital circles, or any subsets of $\mathbb{Z}^{n}$. Also, we mention some examples of digital images that do not have a digital manifold structure. In addition, the important point that we frequently refer to in this section is to reveal the final state of the manifold properties on the discrete setting, which exist in topological spaces but change in digital images. We also deal with some manifold-related notions such as submanifolds, orientations, and partitions of unities. In Section \ref{conc:5}, we present some open problems for many concepts in the previous section and guidance for future digital manifold works.

\section{Preliminaries}
\label{sec:1}

\quad In this section, we introduce the discrete version of some topological concepts such as digital image, adjacency relation, digital homotopy, digital isomorphism, etc., which are the basic elements of digital topology.

\quad Assume that $M$ is a subset of $\mathbb{Z}^{n}$ and $\kappa_{l}$ is a given adjacency relation over the points of $M$. Then a \textit{digital image} is a binary form that works on discrete structures and has an adjacency relation rather than a topology \cite{Kong:1989}. It is generally denoted by the pair $(M,\kappa_{l})$ and the adjacency relation $\kappa_{l}$ is defined as follows:

\begin{definition}\cite{Kong:1989}
	Given any two distinct points $m = (m_{1},\cdots,m_{n})$ and \linebreak$m^{'} = (m^{'}_{1},\cdots,m^{'}_{n})$ of $\mathbb{Z}^{n}$, and an integer $l$ with $1 \leq l \leq n$, we say that $m$ is called \textit{$\kappa_{l}-$adjacent to $m^{'}$} if the following satisfy:
	\begin{itemize}
		\item There exist the maximum number of $l$ indices $k \in \{1,\cdots,n\}$ with the condition $|m_{k}-m^{'}_{k}| = 1$.
		\item The equality $m_{s} = m^{'}_{s}$ holds for any indices $s \in \{1,\cdots,n\}$ with the condition $|m_{s}-m^{'}_{s}| \neq 1$.
	\end{itemize}
\end{definition}

\quad For simple illustrations of adjacency relations, we observe that
\begin{itemize}
	\item $\kappa_{1} = 2-$adjacency in $\mathbb{Z}$,
	\item $\kappa_{1} = 4-$adjacency and $\kappa_{2} = 8-$adjacency in $\mathbb{Z}^{2}$, and
	\item $\kappa_{1} = 6-$adjacency, $\kappa_{2} = 18-$adjacency, and $\kappa_{3} = 26-$adjacency in $\mathbb{Z}^{3}$.
\end{itemize}
The fact that $m$ is $\kappa_{l}-$adjacent to $m^{'}$ is shown by the notation $m \sim_{\kappa_{l}} m^{'}$. 

\quad Assume that $(M,\kappa_{l})$ is a digital image and $m$ is any point of $M$. Then the \textit{digital neighborhood} \cite{Boxer:2017} of $m$ is 
\begin{eqnarray*}
	N_{\kappa_{l}}(m) = \{m^{'} \in \mathbb{Z}^{n} : m \sim_{\kappa_{l}} m^{'}\}.
\end{eqnarray*}
In addition, $N_{\kappa_{l}}^{\ast}(m)$ is given by $N_{\kappa_{l}}(m) \cup \{m\}$.

\quad A digital image $(M,\kappa_{l})$ is called \textit{digitally $\kappa_{l}-$connected} \cite{Herman:1993} if and only if for any distinct points $m$, $m^{'} \in M$, there exists a subset $\{m_{0}, m_{1}, ..., m_{s}\} \subseteq M$ with the following conditions:
\begin{itemize}
	\item The initial and final points are $m$ and $m^{'}$, respectively, i.e., $m = m_{0}$ and $m^{'} = m_{s}$.
	\item One has $m_{j} \sim_{\kappa_{l}} m_{j+1}$ for each $j = 0, 1, ..., s-1$.
\end{itemize} 

\quad By recalling the definition of totally disconnected space, we can give the following in digital images in parallel with topological spaces: 

\begin{definition}
	A digital image is said to be \textit{totally disconnected} provided that it admits only one-point sets as digitally connected subsets.
\end{definition} 

\quad Assume that $\alpha : (M,\kappa_{l}) \rightarrow (M^{'},\kappa_{s})$ is any map of digital images. Then $\alpha$ is \textit{digitally $(\kappa_{l},\kappa_{s})-$continuous} \cite{Boxer:1999,Ros:1986} if
\begin{eqnarray*}
	M_{1} \ \text{is $\kappa_{l}-$connected in} \ M \ \ \Rightarrow \ \ \alpha(M_{1}) \ \text{is $\kappa_{s}-$connected in} \ M^{'}
\end{eqnarray*}
for any subset $M_{1} \subseteq M$. Proposition 2.5 in \cite{Boxer:1999} emphasizes that when we have two digitally continuous maps, their composition is also a digitally continuous map. A \textit{digital $(\kappa_{l},\kappa_{s})-$isomorphism} is a map $\alpha : (M,\kappa_{l}) \rightarrow (M^{'},\kappa_{s})$ of digital images provided that each of the following satisfies \cite{Boxer2:2006,Han:2005}:
\begin{itemize}
	\item $\alpha$ is bijective and $(\kappa_{l},\kappa_{s})-$continuous.
	\item The inverse of $\alpha$ is $(\kappa_{s},\kappa_{l})-$continuous.
\end{itemize} 

\quad The \textit{digital $n$-sphere} \cite{Boxer:2006}, the discrete version of $S^{n}$ in topological spaces, is presented by the boundary of $I_{n}$. More explicitly, $S^{n}$ in digital images is $[-1,1]_{\mathbb{Z}}^{n+1} \setminus \{(0,\cdots,0)\} \subset \mathbb{Z}^{n+1}$, where $(0,\cdots,0)$ denotes the origin of $\mathbb{Z}^{n+1}$ and sometimes shortly expressed by $0_{n+1}$. 

\quad The \textit{strong adjacency} NP, also known as the \textit{normal product adjacency}, is defined as follows for the Cartesian product created by two digital images.

\begin{definition}\cite{Berge:1976}
	Assume that $(M,\kappa_{l})$ and $(M^{'},\kappa_{s})$ are any digital images (or graphs). Let $m_{1}$, $m_{2} \in M$ and $m^{'}_{1}$, $m^{'}_{2} \in M^{'}$ be any points. Then
	\begin{eqnarray*}
		(m_{1},m^{'}_{1}) \sim_{\text{NP}(\kappa_{l},\kappa_{s})} (m_{2},m^{'}_{2}) \ \ \ \text{in} \ \ \ M \times M^{'}
	\end{eqnarray*}
	 if and only if one of the following satisfies:
	\begin{itemize}
		\item[\textbf{a)}] $m_{1} = m_{2}$ and $m^{'}_{1} \sim_{\kappa_{s}} m^{'}_{2}$, or
		\item[\textbf{b)}] $m_{1} \sim_{\kappa_{l}} m_{2}$ and $m^{'}_{1} = m^{'}_{2}$, or
		\item[\textbf{c)}] $m_{1} \sim_{\kappa_{l}} m_{2}$ and $m^{'}_{1} \sim_{\kappa_{s}} m^{'}_{2}$.
	\end{itemize}
\end{definition}

\quad In \cite{BoxKar:2012}, an important property tells us that a $\kappa_{l}-$adjacency can be differ from a strong adjacency. The NP-adjacency is entirely determined by the adjacencies of the factors. Moreover, one has the fact that shows when the NP-adjacency is equal to the $c_{k}-$adjacency.

\begin{proposition}\cite{BoxKar:2012}
	Assume that $(M,\kappa_{l})$ and $(M^{'},\kappa_{s})$ are any digital images. Then the NP-adjacency and the $\kappa_{l+s}-$adjacency are equal to each other in the Cartesian product $M \times M^{'}$.
\end{proposition}

\quad A digital interval \cite{Boxer:2006}, the discrete version of a closed interval in topological spaces, is given by $[m,m^{'}]_{\mathbb{Z}} = \{r \in \mathbb{Z} \ : \ m \leq r \leq m^{'}\}$. Assume that $[0,j]_{\mathbb{Z}}$ is any digital interval with the $2-$adjacency and $f$, $g : (M,\kappa_{l}) \rightarrow (M^{'},\kappa_{s})$ are two $(\kappa_{l},\kappa_{s})-$continuous maps. Then $f$ and $g$ are called \textit{digitally $(\kappa_{l},\kappa_{s})-$homotopic} \cite{Boxer:1999} if there is a digitally continuous map $H : M \times [0,j]_{\mathbb{Z}} \rightarrow M^{'}$ satisfying $H(m,0) = f(x)$ and $H(m,j) = g(x)$ for all $m \in M$ with the properties that, for any fixed $m \in M$, the map $$H_{m} : [0,j]_{\mathbb{Z}} \rightarrow M^{'}$$ defined by $H_{m}(t) = H(m,t)$ is digitally continuous for all $t$, and for any fixed \linebreak$t \in [0,j]_{\mathbb{Z}}$, the map $$H_{t} : M \rightarrow M^{'}$$ defined by $H_{t}(m) = K(m,t)$ is digitally continuous for all $x$. Moreover, $n$ is the number of steps of $H$ and $H$ is called a \textit{digital homotopy} between $f$ and $g$ in $j$ steps. $f \simeq_{\kappa_{l},\kappa_{s}} g$ indicates that $f$ and $g$ are digitally homotopic to each other. Furthermore, a digital image $(M,\kappa_{l})$ is \textit{digitally $\kappa_{l}-$contractible} \cite{Boxer:1999} if $1_{M} \simeq_{\kappa_{l},\kappa_{l}} s$ for some constant map $s : M \rightarrow M$.

\quad Assume that $M^{'}$ is a $\kappa_{s}-$connected digital image, $m_{0} \in M^{'}$ a basepoint, and $f : (M,\kappa_{l}) \rightarrow (M^{'},\kappa_{s})$ a $(\kappa_{l},\kappa_{s})-$continuous map. Then $f$ is a \textit{digital fiber bundle with digital fiber $(N,\kappa_{v})$} provided that the following are satisfied \cite{KaracaVergili:2011}:

\begin{itemize}
	\item $f^{-1}(m_{0}) = N$ and $f$ is surjective.
	\item Let $m^{'} \in M^{'}$ any point. Then there is a $\kappa_{s}-$connected subset $W \subseteq \mathbb{Z}$ with $m^{'} \in W$ and an isomorphism $\varphi : f^{-1}(W) \rightarrow W \times N$ with $\pi_{1} \circ \varphi = f$. Here, $\pi_{1} : (W \times N,\kappa_{*}) \rightarrow (W,\kappa_{s})$ is the first projection, where $\kappa_{*}$ is NP$(\kappa_{s},\kappa_{1})$.
\end{itemize}

\quad The digital Euler characteristic, the discrete version of Euler characteristic in topological spaces, is defined in \cite{Han:2007} but we use the improved version of this notion \cite{KaracaBoxerOztel:2011}: Let $(M,\kappa_{l})$ be a digital image. A set of $i+1$ unique elements of a digital image $(M,\kappa_{l})$, each pair of which is $\kappa_{l}-$adjacent, constitutes a $\kappa_{l}-$simplex in $M$ of dimension $i$. Assume now that $\alpha_{r}$ is the number of unique $r-$dimensional $\kappa_{l}-$simplices in $X$, and $n$ is the greatest integer $i$ such that $M$ has a $\kappa_{l}-$simplex of dimension $i$. Then the \textit{digital Euler characteristic} \cite{KaracaBoxerOztel:2011} of a digital image $(M,\kappa_{l})$, denoted by $\chi(M,\kappa_{l})$ or shortly $\chi(M)$, is defined by the number $\displaystyle \sum_{r=0}^{n}(-1)^{r}\alpha_{r}$.

\begin{proposition}\cite{Han:2007,KaracaBoxerOztel:2011}
	If $(M,\kappa_{l})$ and $(M^{'},\kappa_{s})$ are digitally isomorphic digital images, then their digital Euler characteristics are the same, i.e., $$\chi(M,\kappa_{l}) = \chi(M^{'},\kappa_{s}).$$
\end{proposition}

\section{Digital Manifolds and Some Counterexamples}
\label{sec:2}

\quad A topological $n-$manifold $M$ includes three basic properties as follows.
\begin{itemize}
	\item $M$ is Hausdorff.
	\item $M$ is second-countable.
	\item $M$ admits a local homeomorphism, i.e., every point $m \in M$ has a neighborhood $W$ such that $W$ is homeomorphic to an open subset of $\mathbb{R}^{n}$.  
\end{itemize} 

\quad When defining the concept of a manifold in digital topology is desired, it is necessary to examine these three properties separately on digital images. The properties of being a Hausdorff space and being a second-countable space in digital images do not tell us new things mathematically. We need to focus more on the third property.

\quad The discrete topology is a Hausdorff space since every subset is open. Indeed, any two distinct points can be separated by open sets, each containing only that point. Recall that no topology exists on digital images, that is, every subset can be assumed to be open. Therefore, digital images, the discrete construction of topological spaces over $\mathbb{Z}^{n}$, always tend to have the Hausdorff property.

\quad Let $(M,\kappa)$ be a digital image. Then a collection $\beta$ of subsets of $M$ is called a digital base for the digital image $M$ if and only if every subset $N \subseteq M$ can be written as the union of elements of $\beta$. In addition, a digital image $M$ is called second-countable if $M$ has a countable digital basis. Recall that the discrete topological space $(\mathbb{R},\tau_{d})$ is not a second-countable space even if $\beta = \{m \ : \ m \in \mathbb{R}\}$ is a basis for $\mathbb{R}$. Because $\mathbb{R}$ is not countable. On the other hand, the collection $\{m \ : \ m \in \mathbb{Z}\}$ forms a countable digital basis for the discrete version $\mathbb{Z}$ of real numbers. Thus, digital images, the discrete construction of topological spaces over $\mathbb{Z}^{n}$, always tend to have the second-countable property.

\begin{definition} \label{def1}
	%A digital image $(X,\kappa)$ is called a digital $n-$manifold provided that for a minimum non-negative integer $n$, there is a point $q$ in $\mathbb{Z}^{n}$ such that a nonempty digital image $P \subseteq N_{\lambda}(q)$ is digitally isomorphic to $N_{\kappa}(p)$ of any point $p \in X$.  
	A digital image $(M,\kappa)$ is called a digital $n-$manifold provided that for a minimum non-negative integer $n$, each point $m \in X$ has a digital neighborhood $N_{\kappa}(m)$ that is digitally isomorphic to $N_{\lambda}(s)$ of at least one point $s$ in $\mathbb{Z}^{n}$.  
\end{definition} 

%\quad A digital isomorphism $P \rightarrow N_{\kappa}(p)$ in Definition \ref{def1} is called a digital chart. The inverse of this isomorphism $N_{\kappa}(p) \rightarrow P$ is called a digital local parametrization of $X$ about $p$. A set of digital charts is said to be a digital atlas provided that $X$ can be covered by their domains.

\quad A digital isomorphism $N_{\kappa}(m) \rightarrow N_{\lambda}(s)$ in Definition \ref{def1} is called a digital chart. The inverse of this isomorphism $N_{\lambda}(m) \rightarrow N_{\kappa}(s)$ is called a digital local parametrization of $M$ about $m$. A set of digital charts is said to be a digital atlas provided that $M$ can be covered by their domains.

\begin{example}
	\begin{itemize}
		\item The digital $0-$sphere $S^{0} = \{-1,1\}$ is a digital $0-$manifold since $N_{\kappa}(-1) = N_{\kappa}(1) = \emptyset$ for $-1$ and $1 \in S^0$.   
		
		\item %$\mathbb{Z}$ is a digital $1-$manifold: Let $p$ be any integer. Then $N_{2}(p) = \{p-1,p+1\}$. Therefore, $N_{2}(p)$ is digitally isomorphic to $P = N_{2}(q) = \{q-1,q+1\}$ for any $q$ different from $p$ in $\mathbb{Z}$ by the map $\gamma : N_{2}(p) \rightarrow N_{2}(q)$, $\gamma(p-1) = q-1$ and $\gamma(p+1) = q+1$. 
		$\mathbb{Z}$ is a digital $1-$manifold: Assume that $m$ is any integer. Then we have $N_{2}(m) = \{m-1,m+1\}$. Therefore, $N_{2}(m)$ is digitally isomorphic to $N_{2}(s) = \{s-1,s+1\}$ for any $s$ different from $m$ in $\mathbb{Z}$ by the map \linebreak$\gamma : N_{2}(m) \rightarrow N_{2}(s)$, $\gamma(m-1) = s-1$ and $\gamma(m+1) = s+1$.
		
		\item %$S^{1}$ is a digital $1-$manifold: Assume that $S^{1}$ has $c_{1} = 4$ or $c_{2} = 8-$adjacency. Each point $p \in S^{1}$ has two adjacent points, say $p_{1}$ and $p_{2}$. Therefore, there is a point $q$ in $\mathbb{Z}$ with the property that $P = N_{c_{k}}(p) = \{p_{1},p_{2}\}$ is digitally isomorphic to $N_{2}(q) = \{q-1,q+1\}$ for $k = 1,2$.
		$S^{1}$ is a digital $1-$manifold: Let $S^{1}$ have $\kappa_{1} = 4$ or $\kappa_{2} = 8-$adjacency. Each point $m \in S^{1}$ has two adjacent points, say $m_{1}$ and $m_{2}$. Therefore, there is a point $s$ in $\mathbb{Z}$ with the property that $P = N_{\kappa_{l}}(m) = \{m_{1},m_{2}\}$ is digitally isomorphic to $N_{2}(s) = \{s-1,s+1\}$ for $l = 1,2$.
		
		%$S^{2}$ is a digital $2-$manifold: For any point $p \in S^{2}$, $N_{6}(p)$ has three or four totally disconnected points (for example $(1,1,1)$ has three adjacent points and $(0,0,1)$ has four adjacent points with respect to $6-$adjacency) and $N_{18}(p)$ has exactly four totally disconnected points. When $N_{6}(p)$ has three totally disconnected points, there is a point $q$ in $\mathbb{Z}^{2}$ for which $N_{6}(p)$ is digitally isomorphic to a totally disconnected subset $P_{1} \subset N_{4}(q)$ with three points. If $N_{6}(p)$ has four totally disconnected points, then there is a point $q^{'}$ in $\mathbb{Z}^{2}$ such that $N_{6}(p)$ is digitally isomorphic to a totally disconnected subset $P_{2} = N_{4}(q^{'})$ with four points. Similarly, there is at least one point in $\mathbb{Z}^{2}$ with the property that $N_{18}(p)$ is digitally isomorphic to $P_{2}$ when one admits $18-$adjacency. Furthermore, if we consider $26-$adjacency on $S^{2}$, then $N_{26}(p)$ has four totally disconnected points, too. Thus, by applying a similar process, it is obtained that $S^{2}$ with $26-$adjacency is again a digital $2-$manifold.
	\end{itemize}
\end{example}

\begin{example}\label{exm2}
	%It is not possible to say that any subset of a digital $n-$manifold is a digital $n-$manifold. For example $\mathbb{Z}^{2}$ is a digital $2-$manifold but $$X = \mathbb{Z} \times \{0\} \cup \{0\} \times \mathbb{Z} \subseteq \mathbb{Z}^{2}$$ is not a digital manifold of any dimension: Suppose that $X$ is a digital $2-$manifold with the $4-$adjacency. Then, for any point $p \in X - \{(0,0)\}$, $N_{4}(p)$ has exactly two totally disconnected points (say $p_{1}$ and $p_{2}$ for them), and for the point $p^{'} = (0,0)$, $N_{4}(p^{'})$ has exactly four totally disconnected points. In addition, there is a point $q \in \mathbb{Z}^{2}$ such that $p \in X$ admits a digital isomorphism $\gamma : N_{4}(p) \rightarrow P \subseteq N_{4}(q)$, where $N_{4}(q)$ has exactly four totally disconnected points and $P$ has exactly two totally disconnected points. This contradicts with the fact that $n=2$ is the minimum non-negative integer because there is also a digital isomorphism $$\gamma^{'} : N_{4}(p) = \{p_{1},p_{2}\} \rightarrow N_{2}(r) = \{r-1,r+1\}$$ for at least one point $r \in \mathbb{Z}$ when $n=1$. Moreover, $X$ cannot be a $1-$manifold. Indeed, $N_{4}(p^{'})$ has exactly four points but $P \subseteq \mathbb{Z}$ cannot exceed $2$ points, i.e., a bijective map cannot be constructed. The similar way can be followed when $X$ has the $8-$adjacency. Indeed, $(2,0)$ has exactly two $8-$adjacent points but $(1,0)$ has exactly four $8-$adjacent points.
	It is not possible to say that any subset of a digital $n-$manifold is a digital $n-$manifold. For example $\mathbb{Z}^{2}$ is a digital $2-$manifold but $$M = \mathbb{Z} \times \{0\} \cup \{0\} \times \mathbb{Z} \subseteq \mathbb{Z}^{2}$$ is not a digital manifold of any dimension: Suppose that $M$ is a digital $2-$manifold with the $4-$adjacency. Then, for any point $m \in M - \{(0,0)\}$, we have a digital isomorphism $$\gamma : N_{4}(m) \rightarrow N_{4}(s),$$ where $s \in \mathbb{Z}^{2}$ is any point. However, this is a contradiction because $N_{4}(m)$ has exactly two totally disconnected points and $N_{4}(s)$ has exactly four totally disconnected points, i.e., a bijective map cannot be constructed. Similarly, by assuming that $M$ is a digital $1-$manifold, we have a digital isomorphism $$\gamma^{'} : N_{4}((0,0)) \rightarrow N_{2}(s^{'})$$ for any point $s^{'} \in \mathbb{Z}$. This is again a contradiction because $N_{4}((0,0))$ has four totally disconnected points $(1,0),(0,1),(-1,0)$, and $(0,-1)$ whereas $N_{2}(s^{'})$ has two totally disconnected points $s^{'}-1$ and $s^{'}+1$. Moreover, $M$ cannot be a $1-$manifold. A similar way can be followed when $M$ has the $8-$adjacency. Indeed, $(2,0)$ has exactly two $8-$adjacent points ($(1,0)$ and $(3,0)$) but $(1,0)$ has exactly four $8-$adjacent points ($(2,0)$, $(0,0)$, $(-1,0)$, and $(1,0)$).
\end{example}

%\quad In topological spaces, if $X_{1}$ is a topological $n_{1}-$manifold and $X_{2}$ is a topological $n_{2}-$manifold, then their Cartesian product $X_{1} \times X_{2}$ is a topological $(n_{1} + n_{2})-$manifold. But this is not true for digital images. For example, let $X_{1} = X_{2} = [0,1]_{\mathbb{Z}}$ with $2-$adjacency. Although $X_{i}$ is a digital $1-$manifold for each $i = 1,2$, $X_{1} \times X_{2} = \{(0,0),(0,1),(1,1),(1,0)\}$ is not a digital $2-$manifold with $4-$adjacency. It is again a digital $1-$manifold with $4-$adjacency (note that the Cartesian product is a digital $2-$manifold with $8-$adjacency).

%\quad The union of two digital manifolds need not be a digital manifold again. For instance, let $X$ be the boundary of a digital image $[0,3]_{\mathbb{Z}} \times [0,3]_{\mathbb{Z}}$ which consists of $12$ points and $Y = \{x, y \in \mathbb{Z} : x = y \ \text{and} \ 0 \leq x,y \leq 3\}$ with $4$ points. Assume that both have $4-$adjacency. Even if $X$ is a digital $1-$manifold and $Y$ is a digital $0-$manifold, $X \cup Y$ is not a digital manifold. Indeed, whenever the point $(0,3)$ behaves like a digital $1-$manifold, $(3,3)$ behaves like a digital $2-$manifold and $(2,2)$ behaves like a digital $0-$manifold. A similar example can be observed from the union of digital circles $X = \{(-2,0),(-1,0),(0,0),(0,1),(0,2),(-1,2),(-2,2),(-2,1)\}$ and $Y = \{(0,0),(1,0),(2,0),(2,1),(2,2),(1,2),(0,2),(0,1)\}$: $X$ and $Y$ are digital $1-$manifolds but $X \cup Y$ is not a digital manifold.HATALAR VARRRRR
\quad The union of two digital manifolds need not be a digital manifold again. For instance, let $M_{1}$ and $M_{2}$ be two digital images 
\begin{eqnarray*}
	&&\{(m_{1},0) : 0 \leq m_{1} \leq 4 \} \cup \{(4,m_{1}) : 1 \leq m_{1} \leq 4\}\\
	&&\cup \ \{(m_{1},4) : 0 \leq m_{1} \leq 3\} \cup \{(0,m_{1}) : 1 \leq m_{1} \leq 3\}
\end{eqnarray*}
and $$ \{m_{1}, m_{2} \in \mathbb{Z} : m_{1} = m_{2} \ \text{and} \ 0 \leq m_{1},m_{2} \leq 4\}$$ with the $4-$adjacency, respectively. Even if $M_{1}$ is a digital $1-$manifold and $M_{2}$ is a digital $0-$manifold, $M_{1} \cup M_{2}$ is not a digital manifold. Indeed, whenever the point $(0,4)$ behaves like an element of a digital $1-$manifold ($N_{4}((0,4))$ has two totally disconnected points $(0,3)$ and $(1,4)$) but $(2,2)$ behaves like an element of a digital $0-$manifold ($N_{4}((2,2)) = \emptyset$).

\quad Another difference from a topological space is in the classification of some manifolds in digital images. For instance, a connected $1-$manifold can be homeomorphic to one of the following manifolds: a circle, a closed interval $[0,1]$, the real numbers, or the half-line. In the discrete version, one can intuitively think that a digitally connected $1-$manifold can be digitally isomorphic to a digital interval or a digital $1-$sphere since the discretizations of the closed interval, $\mathbb{R}$, and the half-line are all a digital interval. However, this is false. Indeed, a digital $1-$manifold $M = \{(0,0),(0,1),(1,0),(1,1)\}$ with the $4-$adjacency is not digitally isomorphic to these two possible digital images (see Figure \ref{fig:1}). 

\begin{figure}[h]
	\centering
	\includegraphics[width=0.40\textwidth]{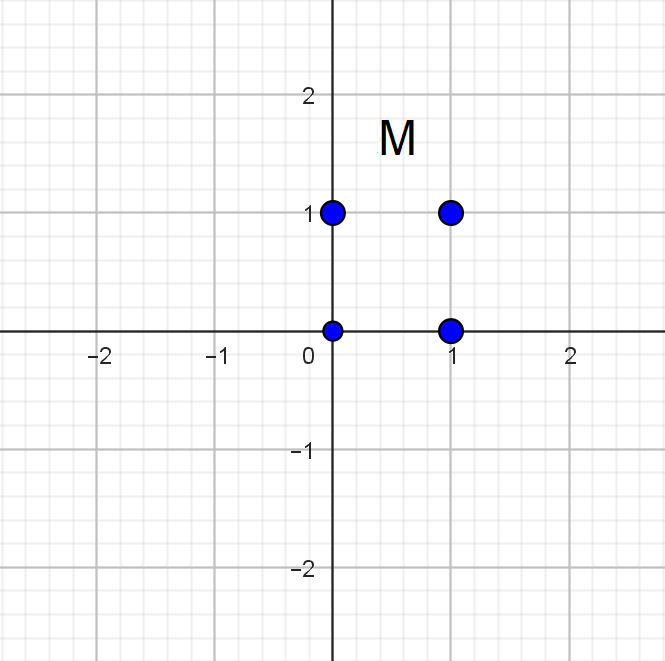}
	\caption{As an example of a digital $1-$manifold, the digital image $M = \{(0,0),(0,1),(1,0),(1,1)\}$ with the $4-$adjacency.}
	\label{fig:1}
\end{figure}

\begin{definition}\label{def2}
	%A digital image $(X,\kappa)$ is called a digital $n-$manifold with boundary provided that for a minimum non-negative integer $n$, there is a point $q$ in $\mathbb{Z}^{n}$ or in a digital half space $D_{+}^{n} = \{(x_{1},\cdots,x_{n}) \in \mathbb{Z}^{n} : x_{n} \geq 0\}$ such that a nonempty digital image $P \subseteq N_{\lambda}(q)$ is digitally isomorphic to $N_{\kappa}(p)$ of any point $p \in X$. 
	A digital image $(M,\kappa)$ is called a digital $n-$manifold with boundary provided that for a minimum non-negative integer $n$, each point $m \in M$ has a digital neighborhood $N_{\kappa}(m)$ that is digitally isomorphic to $N_{\lambda}(s)$ of at least one point $s$ in a digital image $D_{+}^{n} = \{(m_{1},\cdots,m_{n}) \in \mathbb{Z}^{n} : m_{j} \geq 0 \ \text{for all} \ j = 1,\cdots,n\}$ or $\mathbb{Z}^{n}$. 
\end{definition} 

\quad %The boundary and the interior of a digital $n-$manifold $(X,\kappa)$ is defined by
%\begin{eqnarray*}
	%\partial X = \{x \in X : N_{\kappa}(x) \ \text{is digitally isomorphic to} \ P \subseteq N_{\lambda}(q), \exists q \in D_{+}^{n}\}\
%\end{eqnarray*}
%and 
%\begin{eqnarray*}
	%\text{int}(X) = X - \partial X,
%\end{eqnarray*} respectively.
The boundary $\partial M$ of a digital $n-$manifold $(M,\kappa)$ is defined by a digital image having all points mapped to $m_{j} = 0$ by a digital chart for at least one $j \in \{1,\cdots,n\}$ and the interior int$(M)$ of a digital $n-$manifold $(M,\kappa)$ is defined by a digital image having all points mapped to $m_{j} > 0$ by some digital charts for all $j$, i.e.,
\begin{eqnarray*}
 M = \text{int}(M) \cup \partial M.
\end{eqnarray*}

\quad If $\partial M = \emptyset$, then $(M,\kappa)$ is a manifold (without boundary), namely that Definition \ref{def2} corresponds to Definition \ref{def1}.

\begin{example}\label{exm1}
	Let $M = [1,3]_{\mathbb{Z}} \times [1,3]_{\mathbb{Z}}$ be a digital image in $\mathbb{Z}^{2}$ with the $4-$adjacency. The point $(2,2) \in M$ admits a digital isomorphism $$N_{4}((2,2)) \rightarrow N_{4}(0,0)$$ for the point $(0,0) \in D_{+}^{2}$. Some points such as $(2,1)$, $(3,2)$, $(2,3)$, and $(1,2)$ admit a digital isomorphism $$N_{4}(-) \rightarrow N_{4}((1,0))$$ for the point $(1,0) \in D_{+}^{2}$. The points $(1,1)$, $(3,1)$, $(3,3)$, and $(1,3)$ admit a digital isomorphism $$N_{4}(-) \rightarrow N_{4}(5,5)$$ for the point $(5,5) \in D_{+}^{2}$. This shows that $(M,4)$ is a digital $2-$manifold with boundary. Moreover, we observe that $(2,2)$ is the only interior point of $M$ and $\partial M = \{m \in M : m \neq (2,2)\}$ is the boundary set of $M$. It follows that int$(m)$ is a digital $0-$manifold and $\partial M$ is a digital $1-$manifold.
\end{example}

\quad In topological spaces, if $M$ is an $n-$manifold with boundary, then int$(M)$ is an $n-$manifold, too. However, this is not always true for digital images. Example \ref{exm1} presents a counterexample of this in digital images. 

\begin{proposition}
	A digital interval is a digital $1-$manifold with two boundaries.
\end{proposition}

\begin{proof}
	Let $[a,b]_{\mathbb{Z}}$ be any digital interval. Then $N_{2}(a)$ and $N_{2}(b)$ have exactly one totally disconnected points $\{a+1\}$ and $\{b-1\}$, respectively. For any point $0 \in D_{+}^{1}$, we have two digital isomorphisms $$N_{2}(a) \rightarrow N_{2}(0)$$ with $a+1 \mapsto 1$ and $$N_{2}(b) \rightarrow N_{2}(0)$$ with $b-1 \mapsto 1$. It follows that $a$, $b \in \partial M$. Moreover, if $[a,b]_{\mathbb{Z}}$ has any point $c$ such that $c \neq a$ and $c \neq b$, then we have a digital isomorphism $$N_{2}(c) \rightarrow N_{2}(s)$$ with $c-1 \mapsto s-1$ and $c+1 \mapsto s+1$ for any point $s \in D_{+}^{1} - \{0\}$ (namely that \linebreak$c \in$ int$([a,b]_{\mathbb{Z}})$). Therefore, $[a,b]_{\mathbb{Z}}$ is a digital $1-$manifold with two boundaries $a$ and $b$.
\end{proof}

\begin{proposition}\label{teo1}
	The digital $2-$sphere $S^{2}$ is a digital $2-$manifold with boundaries with the $6-$adjacency, and a digital $2-$manifold with the $18$ and $26-$adjacencies, where $D_{+}^{2}$ has the $4-$adjacency.
\end{proposition}

\begin{proof}
	%First, consider the $6-$adjacency. For any point $p \in S^{2}$, $N_{6}(p)$ has three or four totally disconnected points (for example $N_{6}(1,1,1)$ has three totally disconnected points and $N_{6}(0,0,1)$ has four totally disconnected points). Then, for any points $(0,0)$ and $(2,0)$ in $D_{+}^{2}$, there exist two digital isomorphisms
	%\begin{eqnarray*}
		%N_{6}(p_{1}) \rightarrow N_{4}((0,0)) \ \ \text{and} \ \ N_{6}(p_{2}) \rightarrow N_{4}((2,2)),
	%\end{eqnarray*}
    %where $p_{1}$ and $p_{2}$ can be considered any points of $S^{2}$ with the properties that $N_{4}(p_{1})$ has three totally disconnected points and $N_{4}(p_{2})$ has four totally disconnected points, respectively. Thus, $S^{2}$ is a digital $2-$manifold with boundaries. Moreover, $\partial S^{2} = \{(x,y,z) : x,y,z \in \{-1,1\}\}$ and int$(S^{2}) = S^{2} - \partial S^{2}$. Assume now that $S^{2}$ has $18-$adjacency. For each point $p_{3} \in S^{2}$, $N_{18}(p_{3})$ has exactly four totally disconnected points. It follows that $N_{18}(p_{3}) \rightarrow N_{4}(q)$ is a digital isomorphism for any point $q$ in $\mathbb{Z}^{2}$. Furthermore, if we consider $26-$adjacency on $S^{2}$, then each point $p_{4} \in S^{2}$ admits that $N_{26}(p_{4})$ has four totally disconnected points, too. Thus, by applying a similar process, it is obtained that $S^{2}$ with $26-$adjacency is again a digital $2-$manifold.
    First, consider the $6-$adjacency. For any point $m \in S^{2}$, $N_{6}(m)$ has three or four totally disconnected points (for example $N_{6}(1,1,1)$ has three totally disconnected points and $N_{6}(0,0,1)$ has four totally disconnected points). Then, for any points $(1,0)$ and $(2,2)$ in $D_{+}^{2}$, there exist two digital isomorphisms
    \begin{eqnarray*}
    N_{6}(m_{1}) \rightarrow N_{4}((1,0)) \ \ \text{and} \ \ N_{6}(m_{2}) \rightarrow N_{4}((2,2)),
    \end{eqnarray*}
    where $m_{1}$ and $m_{2}$ can be considered any points of $S^{2}$ with the properties that $N_{6}(m_{1})$ has three totally disconnected points and $N_{6}(m_{2})$ has four totally disconnected points, respectively. Thus, $S^{2}$ is a digital $2-$manifold with boundaries. Moreover, $\partial S^{2} = \{(x,y,z) : x,y,z \in \{-1,1\}\}$ and int$(S^{2}) = S^{2} - \partial S^{2}$. Assume now that $S^{2}$ has the $18-$adjacency. For each point $m_{3} \in S^{2}$, $N_{18}(m_{3})$ has exactly four totally disconnected points. It follows that $N_{18}(m_{3}) \rightarrow N_{4}(s)$ is a digital isomorphism for any point $s$ in $\mathbb{Z}^{2}$. Furthermore, if we consider the $26-$adjacency on $S^{2}$, then each point $m_{4} \in S^{2}$ admits that $N_{26}(m_{4})$ has four totally disconnected points, too. Thus, by applying a similar process, it is obtained that $S^{2}$ with the $26-$adjacency is again a digital $2-$manifold.
\end{proof}

\quad In topological spaces, if $M$ is an $n-$manifold with boundary, then $\partial M$ is an $(n-1)-$manifold. However, this is not always true for digital images. Proposition \ref{teo1} presents a counterexample of this in digital images because $(S^{2},6)$ is a digital $2-$manifold with boundary and $\partial S^{2}$ is a digital $0-$manifold, not a digital $1-$manifold, with the $6-$adjacency. 

\begin{proposition}\label{teo2}
	The digital $2-$sphere $S^{2}$ is not a digital $2-$manifold, where $D_{+}^{2}$ has the $8-$adjacency.
\end{proposition}

\begin{proof}
	Let $m_{1} \in \mathbb{Z}$ be nonzero. Then any points of the form $(m_{1},0)$ in $D_{+}^{2}$ admits that $N_{8}((m_{1},0))$ has $5$ adjacent points and any points of the form $(m_{1},m_{2})$ in $D_{+}^{2}$ with $m_{1}$, $m_{2} > 0$ admits that $N_{8}((m_{1},m_{2}))$ has $8$ adjacent points. In addition, $(0,0)$ admits that $N_{8}((0,0))$ has $3$ adjacent points. Since each point $m \in S^{2}$ admits that $N_{6}(m)$ has three or four totally disconnected points, there is no digital isomorphism from $N_{6}(m)$ to $N_{8}(s)$ for any $s \in D_{+}^{2}$ or $s \in \mathbb{Z}^{2}$. In addition, a similar process can be done for $S^{2}$ with the $18$ and $26-$adjacencies.   
\end{proof}

\quad The proofs of Proposition \ref{teo1} and Proposition \ref{teo2} can be generalized over the number of elements for corresponding $M_{\kappa_{l}}(m)$. So, we have the following result:

\begin{theorem}
	\textbf{i)} The digital $n-$sphere is a digital $n-$manifold with boundaries with the $\kappa_{1}-$adjacency, and for other adjacencies, it is a digital $n-$manifold, where $D_{+}^{n}$ has the $\kappa_{1}-$adjacency.
	
	\textbf{ii)} The digital $n-$sphere is not a digital $n-$manifold, where $D_{+}^{2}$ has the $\kappa_{l}-$adjacency with $l > 1$.
\end{theorem}

\quad Note that $[0,1] \times [0,1]$ is a 2-manifold with boundary in topological spaces but $[0,1]_{\mathbb{Z}} \times [0,1]_{\mathbb{Z}}$ is a digital $1-$manifold without boundary with the $4-$adjacency.

\quad Any open subset in $n-$dimensional Euclidean space is a topological $n-$manifold but this cannot be transferred into digital images.

\begin{proposition}
	A digital neighborhood $(N_{\kappa}(m),\kappa_{1})$ of any point $m \in \mathbb{Z}^{n}$ is a digital $0-$manifold. In addition, $(N_{\kappa}^{\ast}(p),\kappa_{l})$ need not be a digital manifold. 
\end{proposition}

\begin{proof}
	%Let $p$ be any point in $(\mathbb{Z}^{n},c_{1})$. Then $N_{c_{1}}(p)$ has totally disconnected $c_{1}$ points. Choose $x$ as one of these points in $N_{c_{1}}(p)$. Since $x$ has no $c_{1}-$neighbours, the digital neighborhood $N_{c_{1}}(p)$ of $p$ is emptyset. This means that it is a digital $0-$manifold. On the other hand, when we consider $(0,0) \in \mathbb{Z}^{2}$, the digital image $N_{4}^{\ast}((0,0)) = \{(0,0),(0,1),(-1,0),(0,-1),(1,0)\}$ is not a digital manifold by Example \ref{exm2}.
	Let $m$ be any point in $(\mathbb{Z}^{n},\kappa_{1})$. Then $N_{\kappa_{1}}(m)$ has totally disconnected $\kappa_{1}$ points. Choose $m^{'}$ as one of these points in $N_{\kappa_{1}}(m)$. Since $m^{'}$ has no $\kappa_{1}-$neighbours, the digital neighborhood $N_{\kappa_{1}}(m)$ of $m$ is emptyset. This means that it is a digital $0-$manifold. On the other hand, when we consider $(0,0) \in \mathbb{Z}^{2}$, the digital image $N_{4}^{\ast}((0,0)) = \{(0,0),(0,1),(-1,0),(0,-1),(1,0)\}$ is not a digital manifold by using the same method with Example \ref{exm2} (observe that $(0,1)$ has only one adjacent point but $(0,0)$ has four adjacent points.).
\end{proof}

\quad For topological manifolds, if it is connected, then its interior is also connected. This is not true for digital images. As for the counterexample, consider the digital image 
\begin{eqnarray*}
	M = ([0,4]_{\mathbb{Z}} \times [0,4]_{\mathbb{Z}}) - \{(2,2)\}
\end{eqnarray*}
with the $4-$adjacency (see Figure \ref{fig:2}). Then int$(M) = \{(1,1),(3,1),(3,3),(1,3)\}$ because the only elements of $M$ that have four adjacent points are $(1,1)$, $(1,3)$, $(3,1)$, and $(3,3)$. It follows that int$(M)$ is not digitally connected even if $M$ is digitally connected $2-$manifold with boundary.

\begin{figure}[h]
	\centering
	\includegraphics[width=0.40\textwidth]{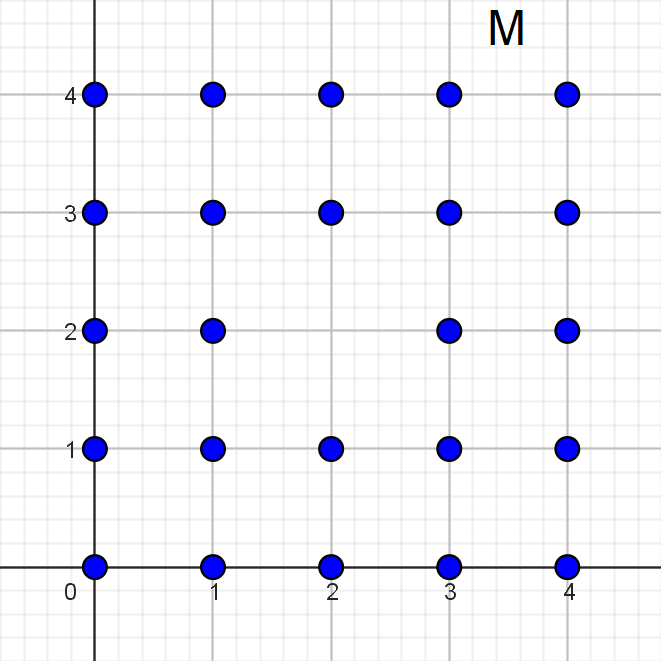}
	\caption{As an example of a digital $2-$manifold with boundary, the digital image $M = ([0,4]_{\mathbb{Z}} \times [0,4]_{\mathbb{Z}}) - \{(2,2)\}$ by considering the $4-$adjacency.}
	\label{fig:2}
\end{figure}

\begin{definition}
	Let $(M,\kappa)$ be a digital $n-$manifold. A subset $S$ of $M$ with the digital $r-$manifold structure for $r \leq n$ is a digital $r-$dimensional submanifold if $S$ is the image of an embedding of digital images, i.e., $\gamma : S \rightarrow M$ is a digital isomorphism onto its image. 
\end{definition}

\quad It is not always the case that $S$ and $M$ have the same dimension. For example, when we consider $\mathbb{Z}^{2}$ as the digital $2-$manifold, $\mathbb{Z}$ is a digital $1-$dimensional submanifold of $\mathbb{Z}^{n}$ by using a digital embedding $\gamma : \mathbb{Z} \rightarrow \mathbb{Z}^{2}$, $\gamma(m) = (m,0)$. 

\begin{example}
	A digital $(n-1)$-sphere is a digital $n-$dimensional submanifold of $\mathbb{Z}^{n}$. 
\end{example}

\quad In topological spaces, $\chi(M_{1} \times M_{2}) = \chi(M_{1})\chi(M_{2})$, where $\chi$ denotes the Euler characteristic of a space. Moreover, $\chi(M_{1}) \sqcup \chi(M_{2}) = \chi(M_{1}) + \chi(M_{2})$, where $\sqcup$ denotes the disjoint union.

\begin{theorem}
	Let $(M_{1},\kappa_{1})$ and $(M_{2},\kappa_{2})$ be any digital manifolds of any dimensions. Then $\chi(M_{1} \times M_{2})$ does not always equal to $\chi(M_{1})\chi(M_{2})$ in digital images. Moreover, $\chi(M_{1} \sqcup M_{2})$ does not always equal to $\chi(M_{1}) + \chi(M_{2})$ in digital images, too.
\end{theorem}

\begin{proof}
	By assuming that $(M_{1},\kappa_{1})$ and $(M_{2},\kappa_{2})$ are both equal to $([0,1]_{\mathbb{Z}},2)$, we have that $\chi(M_{1}) = \chi(M_{2}) = 2 - 1 = 1$. Therefore, $\chi(M_{1})\chi(M_{2}) = 1$. On the other hand, we observe that $\chi(M_{1} \times M_{2},4) = 4 - 4 - 4 = -4$ with the $4-$adjacency. For the second statement, choose $M_{1}$ as $\{(0,0),(1,0)\}$ and $M_{2}$ as $\{(0,1),(1,1)\}$ by considering the $4-$adjacency. Thus, $\chi(M_{1}) + \chi(M_{2}) = 2$ but $\chi(M_{1} \sqcup M_{2}) = -4$.
\end{proof}

\quad Contractible spaces have Euler characteristic $1$ for topological manifolds. In digital images, digital $2-$sphere with the $18-$adjacency is digitally $18-$contractible but its Euler characteristic is equal to $-2$ \cite{KaracaBoxerOztel:2011}.

\quad An orientation of a digital $0-$manifold $M$ is a function $\epsilon : (M,\kappa) \rightarrow (\{\mp 1\},2)$. To compute the number of all orientations in a digital $1-$manifold, a linear order can be used. On a set $M$, the system of rays $\{m \in M \ | \ r < m\}$ for an element $r \in M$ encodes a linear order $<$.

\begin{proposition}
	A digital interval (a type of digital $1-$manifold) has exactly two linear orders.
\end{proposition}

\begin{proof}
	Let $[a,b]_{\mathbb{Z}}$ be any digital interval. Then two linear orders $<$ and $>$ (we can think of these orders as from $a$ to $b$ and from $b$ to $a$, respectively) exist on this interval. It follows that there are two rays $\{p \in M \ | \ r < p\}$ and $\{p \in M \ | \ p < r\}$, the digitally connected components of $M - \{r\}$. When we consider any linear order $<<$ different from $<$, the rays $\{p \in M \ | \ r << p\}$ and $\{p \in M \ | \ p << r\}$ intersect $U$ and $V$ in $M - \{r\}$ in disjoint open sets. Since $U$ and $V$ are digitally connected, one of these rays is equal to $U$ or $V$. This means that $<<$ must be equal to $<$ or $>$. 
\end{proof}

%\quad A local order of a digital $1-$manifold $X$ is a pair $(S,<_{S})$ that consists of a subset $S \subseteq X$ is digitally isomorphic to a digital interval and a linear order $<_{S}$ in $S$. Let $(S,<_{S})$ and $(T,<_{T})$ be two local order of a digital $1-$manifold. Then they agree if $<_{S}$ and $<_{T}$ induce the same order on any digitally connected component $U$ of $S \cap T$. We denote the collection of local orders of a digital $1-$manifold $X$ by LO$(X)$.

%\begin{definition}
%An orientation of a digital $1-$manifold $X$ is a function $\alpha : \text{LO}(X) \rightarrow \{-1,1\}$ satisfying that
%\begin{itemize}
%\item if $\alpha(S,<_{S}) = \alpha(T,<_{T})$, then the restrictions of $<_{S}$ and $<_{T}$ to any digitally connected component $U$ of $S \cap T$ is the same, and
%\item if $\alpha(S,<_{S}) \neq \alpha(T,<_{T})$, then the restrictions of $<_{S}$ and $<_{T}$ to any digitally connected component $U$ of $S \cap T$ is different
%\end{itemize}
%for any local orders $(S,<_{S})$ and $(T,<_{T})$ in $X$.
%\end{definition} 

\quad Thus, we have the following by the previous fact:

\begin{theorem}
	A digital interval has exactly two orientations.
\end{theorem}

\quad Recall that a digital bundle with a digital fiber is given in \cite{KaracaVergili:2011}. A special kind of digital fiber bundle is called a digital vector bundle if the digital fiber is a vector space. The default type of a digital vector bundle is a digital real vector bundle, which has a digital fiber that is a real vector space. Let $M \rightarrow N$ be a digital real vector bundle of rank $s$. Then the $j-$th digital Stiefel-Whitney class is defined as the digital cohomology class $w_{j}(M) \in H^{j,q}(N;\mathbb{Z}_{2})$ for $0 \leq j \leq s$.

\begin{definition}
	A digital manifold $M$ is orientable (otherwise non-orientable) if and only if the first digital Stiefel-Whitney class is zero (otherwise nonzero).
\end{definition}

\begin{example}
	Consider a digital $2-$manifold $$M_{1} = \{(1,1,0),(0,2,0),(-1,1,0),(0,0,0),(0,1,-1),(0,1,1)\}$$ with the $18-$adjacency. By Example 4.2 in \cite{EgeKaraca:2013}, we have that $H^{1,18}(M_{1};\mathbb{Z}_{2}) = 0$. This shows that $w_{1}(M_{1}) = 0$, which concludes that $M_{1}$ is an orientable digital $2-$manifold.
\end{example}

\begin{example}
	Let $M_{2} = [0,1]_{\mathbb{Z}} \times [0,1]_{\mathbb{Z}} \times [0,1]_{\mathbb{Z}}$ be a digital $2-$manifold with boundary with respect to the $6-$adjacency. In \cite{Gulseli:2014}, we know that $H^{1,6}(M_{2};\mathbb{Z}_{2}) = \mathbb{Z}_{2}^{5}$. It follows that $w_{1}(M_{2})$ need not be zero, which concludes that $M_{2}$ is a non-orientable digital $2-$manifold with boundary.
\end{example}

\begin{definition}
	Let $(M,\kappa)$ be a digital image and $f : M \rightarrow \mathbb{Z}$ a function. Then a digital support of $f$ is defined as sp$(f) = \{p \in M : f(p) \neq 0\}$.
\end{definition}

\quad The definition of digital support directly gives some properties:

\begin{itemize}
	\item $p \notin$ sp$(f)$ implies that $f(p)$ must be zero.
	\item The fact sp$(f)$ is the empty set implies that $f$ vanishes. 
\end{itemize} 

\begin{proposition}
	Let $f$, $g : (M,\kappa) \rightarrow (\mathbb{Z},2)$ be two functions. Then
	
	\textbf{a)} sp$(f.g) =$ sp$(f)$ $\cap$ sp$(g)$.
	
	\textbf{b)} sp$(f+g) =$ sp$(f)$ $\cup$ sp$(g)$.
\end{proposition}

\begin{proof}
	Let $p \in \mathbb{Z}$ be any point.
	  
	\textbf{a)} Then we have
	\begin{eqnarray*}
		p \in \text{sp}(f.g) &\Leftrightarrow& (f.g)(p) \neq 0 \Leftrightarrow f(p).g(p) \neq 0 \Leftrightarrow f(p) \neq 0 \ \wedge \ g(p) \neq 0 \\ &\Leftrightarrow& x \in \text{sp}(f) \ \wedge \ p \in \text{sp}(g) \Leftrightarrow p \in \text{sp}(f) \cap \text{sp}(g).
	\end{eqnarray*}
	
	\textbf{b)} Then we have
	\begin{eqnarray*}
		p \in \text{sp}(f+g) &\Leftrightarrow& (f+g)(p) \neq 0 \Leftrightarrow f(p)+g(p) \neq 0 \Leftrightarrow f(p) \neq 0 \ \vee \ g(p) \neq 0 \\ &\Leftrightarrow& p \in \text{sp}(f) \ \vee \ p \in \text{sp}(g) \Leftrightarrow p \in \text{sp}(f) \cup \text{sp}(g).
	\end{eqnarray*}
\end{proof}

\begin{proposition}
	Let $f : (M,\kappa) \rightarrow (\mathbb{Z},2)$ be a function and $\alpha : (M^{'},\kappa^{'}) \rightarrow (M,\kappa)$ a digital isomorphism. Then
	sp$(f \circ \alpha) =$ $\alpha^{-1}(\text{sp}(f))$.
\end{proposition}

\begin{proof}
	Let $p^{'} \in M^{'}$ be any element. Then
	\begin{eqnarray*}
		\alpha (\text{sp}(f \circ \alpha) = \alpha(\{p^{'} \in M^{'} : f(\alpha(p^{'})) \neq 0\}) = \{p \in M: f(p) \neq 0\} = \text{sp}(f). 
	\end{eqnarray*}
    Ths shows that sp$(f \circ \alpha) =$ $\alpha^{-1}(\text{sp}(f))$.
\end{proof}

\begin{definition}
	A digital partition of unity of a digital manifold $(M,\kappa)$ is the set $$\{\alpha_{i} : M \rightarrow \mathbb{Z} \ | \ \alpha_{i} \ \text{is digitally continuous for each} \ i \in [0,m]_{\mathbb{Z}}\}$$ for which the following properties hold:
	\begin{itemize}
		\item[1.] Each $\alpha_{i}$ is nonnegative.
		\item[2.] Every point $p \in M$ admits a digital neighborhood $N_{\kappa}(p)$ with the property that $N_{\kappa}(p) \cap \text{sp}(\alpha_{i}) \neq \emptyset$ for all $i$.
		\item[3.] For each $p \in M$, $g_{1}(p) + \cdots + g_{m}(p) = m$.
	\end{itemize}
\end{definition}

\begin{definition}
	A digital partition of unity $\{\alpha_{i}\}_{i \in [0,m]_{\mathbb{Z}}}$ on a digital manifold $(M,\kappa)$ is subordinate to a cover $\{M_{1},M_{2},\cdots,M_{r}\}$ of $M$ provided that each $\alpha_{i}$ admits a digital image $M_{i}$ of the cover with the property sp$(\alpha_{i}) \subset M_{i}$.
\end{definition}

\begin{example}
	Consider the digital $2-$manifold $\mathbb{Z}^{2}$. Let $m > 1$ be a positive integer and define a cover $\{M_{1},M_{2}\}$ of $\mathbb{Z}^{2}$ by $M_{1} = \{(p_{1},p_{2}) : p_{1} < m\}$ and \linebreak$M_{2} = \{(p_{1},p_{2}) : p_{1} > 0\}$. The intersection of $M_{1}$ and $M_{2}$ is $\{(p_{1},p_{2}) : 0 < p_{1} < m\}$, which is a digital $2-$manifold with boundary. Consider two functions $\alpha_{1} : \mathbb{Z}^{2} \rightarrow \mathbb{Z}$ and $\alpha_{2} : \mathbb{Z}^{2} \rightarrow \mathbb{Z}$ defined by
	\begin{eqnarray*}
		\alpha_{1}(p_{1},p_{2}) = \begin{cases}
			m, & \text{if} \ p_{1} \leq 0 \\
			m-p_{1}, & \text{if} \ 0 < p_{1} < m \\
			0, & \text{if} \ p_{1} \geq m 
		\end{cases}
	\end{eqnarray*}
    and
    \begin{eqnarray*}
    	\alpha_{2}(p_{1},p_{2}) = \begin{cases}
    		0, & \text{if} \ p_{1} \leq 0 \\
    		p_{1}, & \text{if} \ 0 < p_{1} < m \\
    		m, & \text{if} \ p_{1} \geq m 
    	\end{cases},
    \end{eqnarray*}
    respectively. Each $\alpha_{i} \geq 0$ is digitally continuous for $i \in \{1,2\}$ and supported with the corresponding set $M_{i}$. When we consider any point $(p_{1},p_{2}) \in \mathbb{Z}^{2}$, we observe that $\alpha_{1}(p_{1},p_{2}) + \alpha_{2}(p_{1},p_{2}) = m$. If we consider any point $(p_{1},p_{2}) \in M_{1} \cap M_{2}$, then we have that $N_{4}((p_{1},p_{2})) \cap \text{sp}(\alpha_{1})$ and $N_{4}((p_{1},p_{2})) \cap \text{sp}(\alpha_{2})$ are nonempty. Moreover, we get sp$(\alpha_{1}) \subset M_{1}$ and sp$(\alpha_{2}) \subset M_{2}$. Finally, we conclude that $\{\alpha_{1},\alpha_{2}\}$ forms a digital partition of unity on $M_{1} \cap M_{2}$ subordinate to a cover $\{M_{1},M_{2}\}$.  
\end{example}

\section{Conclusion and Open Problems}
\label{conc:5}

\quad Many important homotopy, homology, and cohomology calculations in digital topology are performed on digital curves and surfaces, so studying digital manifolds takes these to the next level. It also opens the way for classifications of digital manifolds, a generalization of digital curves and digital surfaces often used in applications such as animated films, video games, and medical imaging. In topological robotics, a manifold of configurations is often used to model the space in which robots move. Therefore, the robot's journey from the beginning point to the destination point in motion planning difficulties turns into a path planning problem on the manifold. As an out-of-topology example, high-dimensional data often reside on low-dimensional manifolds. Thus, by examining these structures, manifold learning approaches offer data visualization and dimensionality reduction. These illustrations demonstrate the usage of manifolds in a variety of domains, including digital topology, ranging from pure mathematics to applied sciences.
 
\quad The primary goal of this paper is to examine the definition of the term ''digital manifold'' based on digital images, not cells or simplexes. More precisely, the local homeomorphism property is explicitly investigated on digital images based on adjacency relations. After this, the rest of the paper discusses notions, such as submanifold, orientation, and partition of unity, related to the concept of digital manifold. Moreover, some counterexamples of digital images are presented. Besides, a few open problems are left for the readers in this section.

\textbf{Open Problem 1:} Let $(M_{1},\kappa)$ be a digital $n_{1}-$manifold, $(M_{2},\lambda)$ a digital $n_{2}-$manifold, and $n_{3} \leq n_{1}+n_{2}$. Then is it true or not that their Cartesian product $(M_{1} \times M_{2},\text{NP}(\kappa,\lambda))$ is a digital $n_{3}-$manifold? 

\textbf{Open Problem 2:} Are the only digital images to which a digitally connected $1-$manifold is digitally isomorphic the digital $1-$sphere, the digital interval, and the set $M = \{(0,0),(0,1),(1,0),(1,1)\}$? Or are there any other digitally connected $1-$manifolds? Do they always have orientations of $2$?

\textbf{Open Problem 3:} Based on the solution of Open Problem 2, can we say that a digitally connected $1-$manifold is either digitally contractible or digitally homotopy equivalent to a digital $1-$sphere?

\textbf{Open Problem 4:} Is it true that every $n-$dimensional manifold with boundary lies in an $n-$dimensional manifold?

\textbf{Open Problem 5:} Can we define a ''digital smooth (differentiable) manifold''? To answer this, it is valuable to comment on the derivative definitions in digital images. In addition, how do we define the orientation of a digital smooth manifold for $n>1$?

\end{document}